\documentclass[a4paper,12pt,reqno]{amsart}
\usepackage{color}
\usepackage{a4wide}
\usepackage{geometry}
\usepackage{hyperref}
\usepackage{amsmath}
\usepackage{amssymb}
\usepackage{amsthm}
\usepackage[active]{srcltx}
\usepackage[dvips]{graphicx}

\numberwithin{equation}{section}

\newtheorem{thm}{\textbf{Theorem}}[section]

\theoremstyle{remark}
\newtheorem{rem}[thm]{\textbf{Remark}}

\newtheorem{exe}[thm]{\textbf{Example}}

\theoremstyle{definition}

\newtheorem{quest}[thm]{{Question}}
\newtheoremstyle{Claim}{}{}{\itshape}{}{\itshape\bfseries}{:}{ }{#1}
\theoremstyle{Claim}

\newcommand{\R}{{\mathbb R}}

\newcommand{\wt}{\widetilde}
\newcommand{\Int}{\mathrm{Int}}

\makeatletter
\newcommand{\tpitchfork}{%
	\vbox{
		\baselineskip\z@skip
		\lineskip-.52ex
		\lineskiplimit\maxdimen
		\m@th
		\ialign{##\crcr\hidewidth\smash{$-$}\hidewidth\crcr$\pitchfork$\crcr}
	}%
}
\makeatother

\newcommand{\cD}{\mathcal{D}}

\newcommand{\cF}{\mathcal{F}}

\newcommand{\cG}{\mathcal{G}}

\newcommand{\cH}{\mathcal{H}}
\newcommand{\cJ}{\mathcal{J}}
\newcommand{\cK}{\mathcal{K}}

\newcommand{\cM}{\mathcal{M}}
\newcommand{\cN}{\mathcal{N}}
\newcommand{\cP}{\mathcal{P}}

\newcommand{\cS}{\mathcal{S}}

\newcommand{\USC}{\mathrm{USC}}
\newcommand{\LSC}{\mathrm{LSC}}

\newcommand{\changelocaltocdepth}[1]{%
	\addtocontents{toc}{\protect\setcounter{tocdepth}{#1}}%
	\setcounter{tocdepth}{#1}%
}

\title[The Correspondence Principle]{The Correspondence Principle: A bridge between general potential theories and nonlinear elliptic differential operators}
\author[F.R. Harvey]{F. Reese Harvey}
\address{Department of Mathematics\\ Rice University\\ P.O. Box 1892\\ Houston, TX 77005-1892, USA}
\email{harvey@rice.edu (F. Reese Harvey)}
\author[K.R. Payne]{Kevin R. Payne}
\address{Dipartimento di Matematica ``F. Enriques''\\ Universit\`a di Milano\\ Via C. Saldini 50\\ 20133--Milano, Italy}
\email{kevin.payne@unimi.it (Kevin R.\ Payne)}\thanks{Payne partially supported by the Gruppo Nazionale per l'Analisi Matematica, la Probabilit\`a e le loro Applicazioni (GNAMPA) of the Istituto Nazionale di Alta Matematica (INdAM) and the project: GNAMPA 2024 ``Free boundary problems in noncommutative
	structures and degenerate operators''.}

\date{\today} 


\begin{document}

%\dedicatory{ Dedicated to Blaine Lawson on the occasion of his 80th birthday}       

\begin{abstract}
	
General potential theories concern the study of functions which are subharmonic with respect to a suitable constraint set $\cF$ in the space of 2-jets. While interesting in their own right, general potential theories are being widely used to study fully nonlinear PDEs determined by degenerate elliptic operators $F$ acting on the space of 2-jets. We will discuss a powerful tool, the {\em correspondence principle}, which establishes the equivalence between $\cF$–subharmonics/superharmonics $u$ and admissible subsolutions/supersolutions $u$ (in the viscosity sense) of the PDE determined by every operator $F$ which is compatible with $\cF$. The crucial {\em degenerate ellipticity} often requires the operator to be restricted to a suitable constraint set $\cG$, which determines the admissibility. Applications to comparison principles by way of the duality-monotonicity-fiberegularity method will also be discussed. 
	
\end{abstract}

\maketitle

\makeatletter
\def\l@subsection{\@tocline{2}{0pt}{2.5pc}{5pc}{}}
\makeatother

 \setcounter{tocdepth}{1}
\tableofcontents

\changelocaltocdepth{2}

\section{Introduction}\label{sec:intro} 

Our aim in this note is to discuss a bridge between two seemingly different realms. The first realm is that of {\em  general potential theories} which concern the study of generalized {\em subharmonics $u$} with respect to a given {\em constraint set (subequation)} in the space of $2$-jets
	\begin{equation}\label{se}
\cF \subsetneq \cJ^2(X) = X \times (\R \times \R^n \times \cS(n)), \quad X \subset \R^n \ \text{open};
	\end{equation}
as well as generalized {\em superharmonics} and {\em harmonics}. Here $\cS(n)$ is the vector space of all real symmetric $n \times n$ matrices (equipped with the natural partial ordering of the associated quadratic forms). The second realm is that of {\em nonlinear elliptic differential operators} $F : {\rm dom} \, F \subseteq \cJ^2(X) \to \R$, where one is interested in generalized {\em solutions} of the PDE determined by $F$
	\begin{equation}\label{e}
	 F(J^2u):= F(x, u(x), Du(x), D^2u(x)) = 0, \ \ x \in X,
	\end{equation}
as well as generalized {\em subsolutions/supersolutions}. The precise notions of $\cF$-subharmonics/superharmonics will be recalled in section \ref{sec:GPT} and those of subsolutions/supersolutions to \eqref{e} will be recalled in section \ref{sec:PEOs}.

Beginning with the trio of papers of Harvey and Lawson \cite{HL09a}, \cite{HL09b} and \cite{HL09c}, a rich interplay between the two realms is underway. Much can be said about this interplay and the reader is invited to consult the recent survey paper \cite{HP23}. Here we note only a few aspects of this fruitful interplay. On the one hand, the potential theory for a given constraint $\cF$ can simplify, clarify and unify the operator theory for \underline{every} operator $F$ which is {\em compatible} with $\cF$. Moreover, natural results in the potential theory can suggest interesting questions in the operator theory which at first glance might not come to mind. On the other hand, the operator theory of a ``nice'' operator $F$ yields useful information on a potential theory determined by a constraint $\cF$ compatible with $F$. For example, if $F$ is smooth, one can differentiate the equation \eqref{e} to study regularity of $\cF$-harmonics.

In this interplay, an important part of the story is the validity of the {\em correspondence principle}; that is, to determine conditions on an operator constraint pair $(F, \cF)$ for which one has the equivalences
\begin{equation}\label{CP:sub}
u \ \text{is} \  \cF\text{-subharmonic} \ \text{in} \ X	\ \ \Leftrightarrow  \ \  	u \ \text{is a  subsolution of} \ F(J^2u) = 0 \ \text{in} \ X 	.
\end{equation}
\begin{equation}\label{CP:super}
u \ \text{is} \  \cF\text{-superharmonic} \ \text{in} \ X  \ \ \Leftrightarrow  \ \ u \ \text{is a supersolution  of} \ F(J^2u) = 0 \ \text{in} \ X.
\end{equation}

It is worth noting that, in both realms there is a natural Dirichlet problem on domains $\Omega \Subset X$, but assumptions must be made on $\cF$ and $F$ in order to have well defined notions and a robust theory. Of particular interest in this context are: 

\begin{enumerate}
	\item the validity of the {\em comparison principle}
\begin{equation}\label{comparison}
u \leq w \ \text{on} \ \partial \Omega \quad \Rightarrow \quad u \leq w \ \text{on} \ \Omega 
\end{equation}
for each pair $u/w$ of sub/superharmonics for $\cF$ (sub/supersolutions of \eqref{e}), which ensures uniqueness for the Dirichlet problems;
\item  the existence of solutions to the Dirichlet problem by {\em Perron's method}, which also requires suitable boundary geometry of $\Omega$.
\end{enumerate}
If the correspondence principle holds, then resolving (1) and (2) on the potential theory side (for a subequation $\cF$) resolves (1) and (2) for every operator $F$ compatible with $\cF$.

As noted above, assumptions on the pair $F$ and $\cF$ will need to be made in order to have well defined potential theories and operator theories. We will use coordinates $J = (r,p,A) \in \cJ^2 = \R \times \R^n \times \cS(n)$ in the vector space of $2$-jets over each $x \in X \subset\R^n$. We will always assume the related minimal monotonicity properties of {\em degenerate ellipticity}
\begin{equation}\label{DE}
F(x,r,p,A) \leq F(x,r,p, A + P), \quad \forall \, (x,r,p,A) \in {\rm dom} \, F, \forall \, P \geq 0 \ \text{in} \ \cS(n) 
\end{equation}
and {\em positivity}
\begin{equation}
(x,r,p,A) \in \cF \ \Rightarrow \ (x,r,p,A +P)\in \cF, \quad \forall \,  P \geq 0 \ \text{in} \ \cS(n),
\end{equation}
which are needed for the definitions of $u \in \USC(X)$ to be a subsolution of $F(J^2u) = 0$ and $u$ to be $\cF$-subharmonic in $X$, respectively. The definitions are to be taken in the {\em viscosity sense} where one replaces the $2$-jet 
$$
J^2_x u = (u(x), Du(x), D^2 u(x))
$$
(which only makes sense if $u$ is twice differentiable at $x$) with the set 
\begin{equation*}\label{UCJets}
J^{2,+}_{x} u := \{ J^2_{x} \varphi:  \varphi \ \text{is} \ C^2 \ \text{near} \ x, \  u \leq \varphi \ \text{near} \  x \ \text{with equality in} \ x \},
\end{equation*}
of {\em upper test jets}. In order to have the degenerate ellipticity \eqref{DE}, one may need to restrict the domain of $F$, as happens for the Monge-Amp\`ere operator $F(D^2u) = {\rm det}(D^2u)$, where one must restict $F$ to the set convex functions on $X$. These needed assumptions on $F$ and $\cF$  will be discussed in more detail in sections \ref{sec:GPT} and \ref{sec:PEOs}, but the complete program will require sufficient {\em  monotonicity} and sufficient {\em regularity}. The regularity will always be sufficient in the {\em constant coefficient case} and the monotonicity will always be sufficient for $\cM$-monotone subequations $\cF$ and $\cM$-monotone operators $F$ with respect to a (constant coefficient) {\em monotonicity cone subequation} $\cM \subset \cJ^2 := \R \times \R^n \times \cS(n)$. Moreover, two additional (and natural) relations much be satisfied and are called the {\em correspondence relation} and the {\em compatibility condition}.

Before summarizing the contents of this note, we make some remarks on the evolution of the correspondence principle. The first explicit version of the correspondence principle was proven in the constant coefficient setting in \cite{CHLP23} under the assumption of $\cM$-monotonicity for the pair $(F, \cF)$, which is automatic in the gradient-free case. The extension to the variable coefficient setting was proven in \cite{CPR23} under the additional regularity assumption of  {\em fiberegularity} of the pair $(F, \cF)$. The fiberegularity is automatic in the constant coefficient setting. 

An important special case of this result was implicit in \cite{HL19a}, which studied inhomogeneous equations of the form
\begin{equation}\label{IHE}
G(D^2_x u) = f(x), \quad x \in \Omega \subset \R^n
\end{equation}
with $G$ a degenerate elliptic operator with subequation domain $\mathcal{A} \subset \cS(n)$ with the crucial assumption that $G$ is {\em topologically tame}.  This means that the level sets of $G$ have empty interior; this is implied by by an minimal strict monotonicity of $G$ in $A \in \mathcal{A}$. This was refined in an explicit way in \cite{HP25} for obtaining pointwise {\em apriori} estimates for viscosity solutions \eqref{IHE} where $G$ is an {\em $I$-central G\aa rding-Dirichlet polynomial} of degree $N$ on $\cS(n)$ and $f$ is continuous and non-negative. The correspondence principle in this context is proven in Theorem 4.4 of \cite{HP25} and concerns the (pure second order) operator-subequation pair
$$
F(x,A) = G(A) - f(x) \quad \text{and} \quad \cF = \{(x,A): \ A \in \overline{\Gamma} \subset \cS(n): G(A) - f(x) \geq 0\},
$$
where $\overline{\Gamma}$ is the (closed) {\em G\aa rding cone} associated to $G$. An interesting point in this analysis is how the {\em tameness property} of $G$ ensures the  fiberegularity of $(F, \cF)$ for each continuous $f$ (see Remark 4.3 of \cite{HP25}).

Earlier, more rudimentary versions of the correspondence principle were proven in \cite{CP17} and \cite{CP21} in the variable coefficient pure second order and gradient-free settings, respectively. In both of these earlier variable coefficient papers, the additional regularity assumption was already formulated, but a more cumbersome approach was taken which followed more closely some seminal ideas in the important paper of Krylov \cite{Kv95}.

We now descibe briefly the contents of this paper. We begin, in section \ref{sec:CPT}, with two simple but suggestive examples of the correspondence principle in the pure second order constant coefficient setting. These two examples concern the standard potential theory associated to the Laplace operator and the potential theory of convex functions associated to the minimal eigenvalue operator. 

In section \ref{sec:GPT}, we recall the basic notions from general potential theories, with an emphasis on the key notions of duality, monotonicity and fiberegularity. These ingredients give rise to a general and elegant method for proving the comparison principle \eqref{comparison} in the potential theoretic realm. 

In section \ref{sec:PEOs}, a class of nonlinear elliptic operators is recalled. These operators are called {\em proper elliptic} if (suitable restrictions of them) are continuous and have the two monotonicity properties of being {\em degenerate elliptic} as in \eqref{DE} and {\em proper} in the sense that 
\begin{equation}\label{proper}
F(x,r,p,A) \leq F(x,r + s,p, A), \quad \forall \, (x,r,p,A) \in \cG, \forall \, s \leq 0 \ \text{in} \ \R,
\end{equation}
which is needed to avoid obvious counterexamples to the comparison (and maximum) principles for $F$.
There are two cases: the {\em constrained case} in which $F$ is proper elliptic with $\cG = \cJ^2(X)$ (no restriction of the domain of $F$ is needed) and the {\em unconstrained case} in which $F$ must be restricted to a domain $\cG$ which is a subequation. In particular, the notions of {\em $\cG$-admissible subsolutions/supersolutions of $F(J^2u)=0$} are given.

In section \ref{sec:TCP}, we discuss the General Correspondence Principle (Theorem \ref{thm:GCP}) for a proper elliptic operator $F \in C(\cG)$ and the constraint set defined by the {\em correspondence relation}
\begin{equation}\label{CR:intro}
\cF:= \{ (x,J) \in \cG: \ F(x,J) \geq 0\}.
\end{equation}
The theorem states that if
\begin{equation}\label{Fsubequation}
\cF \ \text{defined by} \ \eqref{CR:intro} \ \text{is a subequation}
\end{equation}
and $(F, \cF)$ satisfies the {\em compatibility condition} \footnote{Here and below, given $\cF \subset \cJ^2(X)$, ${\rm Int} \, \cF$ is the interior of $\cF$ and $\cF^{c} = \cJ^2(X) \setminus \cF$ is the complement of $\cF$.}
\begin{equation}\label{CC:intro}
\Int \cF = \{ (x,J) \in \cG: \ F(x,J) > 0 \},
\end{equation}
then the correspondence principle holds; that is, on $X$:
	\begin{equation*}
u \ \text{is} \  \cF\text{-subharmonic} 	\ \ \Leftrightarrow  \ \  	u \ \text{is a $\cG$-admissible subsolution of} \ F(J^2u) = 0;
\end{equation*}
\begin{equation*}
u \ \text{is} \  \cF\text{-superharmonic} \ \ \Leftrightarrow  \ \ u \ \text{is a $\cG$-admissible supersolution  of} \ F(J^2u) = 0.
\end{equation*}
A proof of the theorem is given for the convenience of the reader, which stresses the importance of the assumptions made. In order to apply the theorem to a given proper elliptic operator $F \in C(\cG)$, one needs to check that the candidate subequation $\cF$ defined by the correspondence relation \eqref{CR:intro} is indeed a subequation and that the compatibility condition \eqref{CC:intro} is satisfied. Sufficient conditions for \eqref{Fsubequation} are the $\cM$-monotonicity and fiberegularity of the proper elliptic pair $(F, \cG)$. This is discussed in Theorem \ref{thm:FRMM}, which says that not only is $\cF$ defined by \eqref{CR:intro} a subequation, it is also $\cM$-monotone and fiberegular, and hence the monotonicity-duality-fiberegularity method  applies for proving the comparison principle. Finally, we address the question of whether a given proper elliptic operator $F \in C(\cG)$ admits a subequation $\cF$ that satisfies the compatibility condition \eqref{CC:intro}. We present three examples that show that this can be easy, impossible or requires some work.

\section{Two simple examples of the correspondence principle}\label{sec:CPT}

As a warm-up to the general correspondence principle, we consider two simple but representative examples of the correspondence principle. We start with the most classical example of the potential theory of subharmonic functions determined by the Laplace operator and then the potential theory of convex functions determined by the minimal eigenvalue operator. Both are examples of pure second order potential theories with constant coefficients, which allows for a streamlined notational treatment by ignoring the ``silent variables'' $(x,r,p) \in X \times \R \times \R^n$ in the space of $2$-jets over an open subset $X \subset \R^n$.

\subsection{The correspondence principle in Lapalcian potential theory}

We discuss two versions of the correspondence principle between sub/supersolutions of Laplace's equation and conventional sub/superharmonic functions. The two versions differ only in the regularity of the functions in play; one for for classical ($C^2$) regularity and the other for mere semicontinuity (viscosity notions). The passage from the first to the second is instructive as it highlights the importance of degenerate ellipticity for operators and positivity for constraint sets in a potential theory. 

Consider the {\em Laplace  operator}  $F : \cS(n) \to \R$ defined by $F(A) := {\rm tr} \, A$ which determines the Laplace equation
\begin{equation}\label{LE}
F(D^2u) = {\rm tr} \, D^2u = \Delta u = 0 \quad \text{in} \ X \subset \R^n \ \text{open}.
\end{equation}
Notice that $F$ has constant coefficients and is a pure second order operator. Naturally associated to $F$ is the pure second order constant coefficient {\em subharmonic constraint set} 
\begin{equation}\label{SH}
 \cH := \{ A \in \cS(n): \ \ F(A) = {\rm tr} \, A \geq 0 \}.
\end{equation}

\noindent{\bf Classical formulation:}
For $u \in C^2(X)$ one has the obvious equivalences between the differential inequality that defines classical subsolutions of \eqref{LE} and the differential inclusion that defines classical $\cH$-subharmonics; that is,
	\begin{equation}\label{Laplace1}
F(D^2u(x)) \geq 0 \quad \Leftrightarrow \quad 	D^2u(x) \in \cH 	, \quad \forall \, x \in X
	\end{equation}
and between classical supersolutions of \eqref{LE} and classical $\cH$-superharmonics
	\begin{equation}\label{Laplace2}
F(D^2u(x)) \leq 0 \quad \Leftrightarrow \quad 	D^2u(x) \in -\cH 	, \quad \forall \, x \in X.
	\end{equation}
Since $D^2(-u) = - D^2 u$,  one has that $u$ is $\cH$-superharmonic if and only if $-u$ is $\cH$-subharmonic where $\cH$ is {\em self-dual} in the sense that \footnote{See subsection \ref{sec:duality} for the general notion of duality.}  
$$
\cH = \wt{\cH} := - ( {\rm Int} \, \cH)^c
$$ 
and hence 
$$
\mbox{$u \in C^2(X)$ is $\cH$-superharmonic in $X \ \Leftrightarrow \ -u$ is $\wt{\cH}$-subharmonic in $X$}.
$$
The equivalences above are the correspondence principle for the Laplace operator and the harmonic constraint set for $C^2$ functions.

Notice that the pair $(F, \cH)$ has the related monotonicity properties:
\begin{equation*}
F(A) \leq F(A + P), \quad \forall \, A \in \cS(n), \forall \, P \geq 0 \ \text{in} \ \cS(n) \quad \text{(Degenerate Ellipticity)}
\end{equation*}
\begin{equation*}
A \in \cH \ \Rightarrow \ A + P \in \cH, \quad \forall \,  P \geq 0 \ \text{in} \ \cS(n) \quad \text{(Positivity)}
\end{equation*}
These monotonicity properties ensure that the correspondence principle above for $u \in C^2(X)$ extends to semicontinuous functions by using viscosity notions. 

\noindent {\bf Viscosity formulation:} For the viscosity calculus, it is convenient to consider 
%\note{Put a footnote about why? Weierstrass, extension by $-\infty$?}  
semicontinuous functions taking values in the extended reals $[-\infty, \infty]$, more precisely we will define
$$
\USC(X) := \{ u : X \to [-\infty, + \infty): u(x) \geq \limsup_{X \ni y \to x} u(y), \quad \forall \, x \in X \}
$$
and
$$
\LSC(X):= \{ u : X \to (-\infty, + \infty]: u(x) \leq \liminf_{X \ni y \to x} u(y), \quad \forall x \in X \}.
$$
The basic idea of the viscosity formulation is the following.  In the differential inequalities and differential inclusions for operators and constraints, $\forall\, x \in X$ replace the Hessian $D^2u(x)$ (which only makes sense for $u$ twice differentiable) with 
$$
D^{2,\pm}_x u := \{ D^2 \varphi(x): \  \varphi \ \text{is $C^2$ near $x, u {\leq \above 0pt \geq} \varphi$  near $x$, and $u(x) = \varphi(x)\}$};
$$
the set of Hessians of {\em upper/lower contact functions $\varphi$ for $u$ at $x$}. One clearly has the two equivalences: for each $u \in \USC(X)$ and each $x \in X$
\begin{equation}\label{Laplace3}
	F(A) = {\rm tr} \, A \geq 0, \ \ \forall \, A \in D^{2,+}_x u \quad \Leftrightarrow \quad D^{2,+}_x u \subset \cH.
	\end{equation}
and for each  $u \in \LSC(X)$ and each $x \in X$
	\begin{equation}\label{Laplace4}
	F(A) = {\rm tr} \, A \leq 0, \ \ \forall \, A \in D^{2,-}_x u \ \  \Leftrightarrow \ \ D^{2,-}_x u \subset -\cH  \ \ \Leftrightarrow \quad D^{2,+}_x (-u) \subset \wt{\cH}.
	\end{equation}
The validity of \eqref{Laplace3} and \eqref{Laplace4} is the correspondence principle in the conventional Laplacian case. 

Two remarks are worth recording. %\note{Say something also about equivalent jet formulations/reduction?}

\begin{rem}
It is straightforward to check that the properties \eqref{Laplace3} and \eqref{Laplace4} are equivalent to the validity of the classical mean value inequalities on balls and spheres contained in $X$ 
%\note{Make this explicit?}
\end{rem}

\begin{rem}\label{rem:positivity}
The related monotonicity properties of degenerate ellipticity for $F$ and positivity for $\cH$ are necessary in order for the viscosity notions to be well defined. Indeed, suppose that $u$ is a viscosity subsolution for $F(D^2u) = 0$, but that degenerate ellipticity fails for $F$. Then for some $A \in \cS(n)$ and in some direction $P \in \cS(n)$ one has $F(A) > F(A + P)$. 
This is a problem if $F(A) = 0$ since, for each upper contact function $\varphi$ for $u$ at $x_0$ with $D^2 \varphi(x_0) = A$,
$$
\wt{\varphi}(x):= \varphi(x) + \frac{1}{2} \langle P(x - x_0), x - x_0 \rangle
$$
is also an upper contact function for $u$ at $x_0$ but 
$$
F((D^2\wt{\varphi})(x_0)) = F(A + P) < F(A) = 0,
$$
contradicting the the fact that $u$ is a viscosity subsolution.
\end{rem}

\subsection{The correspondence principle in convex potential theory}

Consider the {\em convexity constraint set}:
\begin{equation}\label{convexity}
	\cP:= \{A \in \cS(n): \ A \geq 0\}.
\end{equation}
As shown in \cite{HL09c}, the set $\cP(X) \subset \USC(X)$ of $\cP$-subharmonics on $X$ are the convex functions (on the connected components where $u \neq -\infty$). Moreover, the constraint set $\cP$ can be described in terms of the  {\em minimal eigenvalue operator} 
\begin{equation}\label{MEO}
\cP = \{ A \in \cS(n): \ F(A) := \lambda_{\rm min}(A) \geq 0 \ \text{in} \ \cS(n)\}.
\end{equation}
This easily yields the correspondence principle for the operator-constraint set pair $(\lambda_{\rm min}, \cP)$; namely for each $u \in \USC(X)$ and each $x \in X$
\begin{equation}\label{convex1}
\lambda_{\rm min}(A) \geq 0, \quad \forall \, A \in D^{2,+}_x u \quad \Leftrightarrow \quad D^{2,+}_x u \subset  \cP
\end{equation}
and for each $u \in \LSC(X)$ and each $x \in X$
\begin{equation}\label{convex2}
\lambda_{\rm min}(A) \leq 0, \quad \forall \, A \in D^{2,-}_x u \quad \Leftrightarrow \quad D^{2,-}_x u \subset - ({\rm Int} \, \cP)^{\circ}.
\end{equation}
Taken together \eqref{convex1} and \eqref{convex2} say that for $u \in \USC(X)$
\begin{equation}\label{convex3}
\mbox{$u$ is $\cP$-subharmonic in $X \ \ \Leftrightarrow \ \ u$ is a subsolution of $\lambda_{\rm min}(D^2u) = 0$ in $X$}
\end{equation}
and for $u \in \LSC(X)$ 
\begin{equation}\label{convex4}
\mbox{$u$ is a supersolution of $\lambda_{\rm min}(D^2u) = 0$ in $X \ \ \Leftrightarrow \ \ u$ is $\cP$-superharmonic in $X$.}
\end{equation}

The ``super'' part of the corresponence principle has an equivalent reformulation by exploiting duality. Consider the dual constraint set
\begin{equation}\label{subaffine}
\wt{\cP} := -( {\rm Int} \, \cP)^c = \{ A \in \cS(n): \lambda_{\rm max}(A) \geq 0\}
\end{equation}
is called the {\em subaffine constraint set} because $u \in \USC(X)$ is $\wt{\cP}$-subharmonic if and only if it is subaffine; that is, it satisfies the following comparison principle
	$$
	u \leq a \ \text{on} \ \partial \Omega \ \Rightarrow \ u \leq a \ \text{on} \  \Omega, \quad \forall \,  \ \text{affine on $\R^n$}, \ \forall \, \Omega \Subset X.
	$$
In the  ``super'' part of the correspondence principle \eqref{convex2}, for each $u \in \LSC(X)$ and for each $x \in X$ one clearly has
$$
\lambda_{\rm min}(A) \leq 0, \quad \forall \, A \in D^{2,-}_x u \quad \Leftrightarrow \lambda_{\rm max}(A) \geq 0, \ \ \forall \, A \in D^{2,+}_x u,
$$
which in light of \eqref{subaffine} yields
\begin{equation}\label{convex5}
\mbox{$u$ is a supersolution of $\lambda_{\rm min}(D^2u) = 0$ in $X \ \ \Leftrightarrow \ \ -u$ is $\wt{\cP}$-superharmonic in $X$.}
\end{equation}

\section{General potential theories}\label{sec:GPT}

In this section, for the benefit of the reader, we give a brief description of the notions from general potential theories that will be used in the general correspondence principle of section \ref{sec:TCP}. We will adopt an informal style; avoiding a long sequence of formal definitions 
%\note{Is this a good idea, or better to give a few formal definitions?} 
but providing a precise description of what is needed. In all that follows, $X$ is an open subset of $\R^n$ with the $2$-jet space 
$$
\cJ^2(X) = X \times (\R \times \R^n \times \cS(n)) = X \times \cJ^2,
$$
as presented at the beginning of the introduction, and we adopt jet coordinates
$$
J = (r,p,A) \quad \text{for each} \ J \in \cJ^2.
$$
For a subset $\cF \subset \cJ^2(X)$ we will denote by
$$
\cF_x := \{ J  \in \cJ^2: (x,J) \in \cF\}, \quad x \in X
$$
the {\em fiber of $\cF$ over $x$}.

\subsection{Subequations.} This is the class of ``good'' constraint sets $\cF \subset \cJ^2(X)$ with a robust potential theory. There are three axioms. The first two are pointwise monotonicity properties; namely, the {\em positivity} condition
\begin{equation*}\label{P}\tag{P}
\forall \, x \in X: \quad (r,p,A) \in \cF_x \ \ \Rightarrow \ \ (r,p,A + P) \in \cF_x, \ \ \forall \, P \geq 0 \ \text{in} \ \cS(n),
\end{equation*}
and the {\em negativity} condition
\begin{equation*}\label{N}\tag{N}
\forall \, x \in X: \quad (r,p,A) \in \cF_x \ \ \Rightarrow \ \ (r,p + s,A) \in \cF_x, \ \ \forall \, s \leq 0 \ \text{in} \ \R.
\end{equation*}
The third axiom is the trio of {\em topological stability} conditions
\begin{equation*}\label{T}\tag{T}
	\cF = \overline{\Int\,  \cF}, \quad \left( \Int \, \cF \right)_x = \Int \, (\cF_x) , \quad \text{and} \quad  \cF_x = \overline{\Int \, \left( \cF_x \right)}, \quad \forall \, x \in X.
\end{equation*}
Notice that by the first property in (T), $\cF$ is closed in $\cJ^2(X)$ and each fiber $\cF_x$ is closed in $\cJ^2$ by the third property in (T). %\note{Say something about non-empty and proper fibers?}
%In addition, the interesting case is when each fiber $\cF_x$ is not all of $\cJ^2$, which we almost always assume. 
Also notice that in the constant coefficient pure second order case $\cF \subset \cS(n)$, as discsussed in section \ref{sec:CPT}, property (N) is automatic and property (T) reduces to the first property $\cF = \overline{\Int \, \cF}$, which is implied by (P) for $\cF$ closed. Hence in this case subequations $\cF \subset \cS(n)$ are closed sets simply satisfying (P).

The conditions (P), (T) and (N) have various (important) implications for the potential theory determined by $\cF$. Some of these will be mentioned at the end of the next subsection.

\subsection{Subharmonics}\label{sec:sub} Given a subequation $\cF$, there are two different formulations for its subharmonics at differing degrees of regularity. A function $u \in C^2(X)$ is {\em $\cF$-subharmononic on $X$} if
\begin{equation}\label{sub1}
J^2_x u = (u(x), Du(x), D^2u(x)) \in  \cF_x, \ \ \forall \, x \in X.
\end{equation}
For $u \in \USC(X)$, one replaces the two jet $J^2_x u$ with the set of {\em upper contact test jets for $u$ at $x$}
\begin{equation}\label{UCJ}
J^{2,+}_{x} u := \{ J^2_{x} \varphi:  \varphi \ \text{is} \ C^2 \ \text{near} \ x, \  u \leq \varphi \ \text{near} \  x \ \text{with equality in} \ x \},
\end{equation}
which are the $2$-jets of {\em upper contact functions $\varphi$ at $x$}. One says that  $u \in \USC(X)$ is {\em $\cF$-subharmononic on $X$} if
\begin{equation}\label{sub2}
J^{2,+}_x u \subset  \cF_x, \ \ \forall \, x \in X,
\end{equation}
and one denotes by $\cF(X)$ the space of $\cF$-subharmononics on $X$.

We now recall some of the implications that properties (P), (T) and (N) have on an $\cF$-potential theory. Property (P) is crucial for the well-posedness of the (viscosity) notion of $\cF$-subharmonicity, as discussed in Remark \ref{rem:positivity}. But it is also needed for {\em $C^2$-coherence}, meaning classical $\cF$-subharmonics are $\cF$-subharmonics in the sense \eqref{sub2}, since for $u$ which is $C^2$ near $x$, one has
$$
J^{2,+}_x u = J_x^2u + (\{0\} \times \{0\} \times \cP) \ \ \text{where} \ \ \cP = \{ P \in \cS(n): \ P \geq 0 \}.
$$
Next note that property (T) insures the local existence of strict classical $\cF$-superharmonics at points $x \in X$ for which $\cF_x$ is non-empty. One simply takes the quadratic polynomial whose $2$-jet at $x$ is $J \in \Int \, (\cF_x)$. Finally, property
(N) eliminates obvious counterexamples to comparison. The simplest counterexample is provided by the constraint set $\cF \subset J^2(\R)$ in dimension one associated to the equation $u^{\prime \prime} - u = 0$.

\subsection{Duality and superharmonics.}\label{sec:duality} A natural notion for a function $w \in \LSC(X)$ to be {\em $\cF$-superharmononic on $X$} is
\begin{equation}\label{super1}
J^{2,-}_x w \subset  \left( {\rm Int} \, \cF_x \right)^c, \ \ \forall \, x \in X \quad  \text{(complement in $\cJ^2$}),
\end{equation}
where 
\begin{equation}\label{LCJ}
J^{2,-}_{x} w := \{ J^2_{x} \varphi:  \varphi \ \text{is} \ C^2 \ \text{near} \ x, \  w \geq \varphi \ \text{near} \  x \ \text{with equality in} \ x \},
\end{equation}
is the space of {\em lower contact jets for $w$ at $x$}. However, for many reasons, it is better to reformulate \eqref{super1} using the notion of {\em duality}. For a given subequation $\cF \subset \cJ^2(X)$ the {\em Dirichlet dual} of $\cF$ is the set $\wt{\cF} \subset \cJ^2(X)$ given by 
\begin{equation}\label{dual}
\wt{\cF} :=  (- \Int \, \cF)^c = - (\Int \, \cF)^c \ \ \text{(complement in 
	$\cJ^2(X)$)}.
\end{equation}
With the help of property (T), the dual can be calculated fiberwise
\begin{equation}\label{dual_fiber}
\wt{\cF}_x :=  (- \Int \, (\cF_x))^c = - (\Int \, (\cF_x))^c, \ \ \forall \, x \in X \quad \text{(complement in  $\cJ^2$)}.
\end{equation}
This is a true duality in the sense that one can show
\begin{equation}\label{true_dual}
\wt{\wt{\cF}} = \cF \quad \text{and} \quad \text{$\cF$ is a subequation \ \ $\Leftrightarrow$ \ \  $\wt{\cF}$ is a subequation}. 
\end{equation}

\begin{rem}[Duality reformulation of superharmonics]\label{rem:super} Making use of duality \eqref{dual} and the fact that $J^{2,+}_x(-w) = -J^{2,-}_x w$, it is clear that  
\begin{equation}\label{super2}
 w \ \text{is \ $\cF$-superharmonic on $X$} \Leftrightarrow -w \in \wt{\cF}(X).
 \end{equation}
\end{rem}

Two additional notions will play a role in the proof of both general correspondence principles and general comparison principles. We will discuss them next.

\subsection{Monotonicity.}\label{sec:monotonicity} Given a subset  $\cM \subset \cJ^2$, a  subequation $\cF \subset \cJ^2(X)$ is {\em $\cM$-monotone} if
\begin{equation}\label{mono1}
\cF_x + \cM \subset \cF_x, \ \ \forall \, x \in X.
\end{equation}
Notice that combining the properties (P) and (N) for any subequation is equivalent to $\cF$ being $\cM_0$-monotone for the {\em minimal monotonicity cone} in $\cJ^2$ defined by
$$
\cM_0: = \cN \times \{0\} \times \cP = \{(s,0,P) \in \cJ^2: \ \text{$s \leq 0$ in $\R$ and $P \geq 0$ in $\cS(n)$} \}.
$$
$\cM_0$ satisfies properties (P) and (N) (it is $\cM_0$-monotone) but fails to be a subequation since it has empty interior and hence property (T) fails. 

Consider now the situation in which $\cF$ has additional monotonticity in the sense that it is $\cM$-monotone for a {\em monotonicity cone subequation} $\cM$ (a subequation which is also closed convex cone with vertex at the origin). Notice that if $\cF$ is {\em gradient-free} then this always happens because
$$
\cM(\cN, \cP):= \cN \times \R^n \times \cP
$$
is a monotonicity cone subequation for $\cF$.

If $\cF$ admits a monotonicity cone subequation $\cM$, then, using duality, one has the following important equivalence of {\em jet addition formulas} describing $\cM$-monotonicity
\begin{equation}\label{jaf}
\forall \, x \in X: \quad \cF_x + \cM \subset \cF_x \ \ \Leftrightarrow \ \ \cF_x + \wt{\cF}_x \subset \wt{\cM}.
\end{equation}
Since the dual cone $\wt{\cM}$ is also a subequation, it determines a (constant coefficient) potential theory. If the algebraic statement in the jet addition formula on the right of \eqref{jaf} yields an analytic statement in the form of a {\em subharmonic addition formula}
	\begin{equation}\label{saf}
u \in \cF(X), v \in \wt{\cF}(X) \ \ \ \Rightarrow \ \ u + v \in \wt{\cM}(X),
\end{equation}
then one has a clear strategy for proving the {\em comparison principle} on $\Omega \Subset X$:
\begin{equation}\label{CP1}
u \leq w \ \text{on} \ \partial \Omega \ \ \Rightarrow \ \ u \leq w \ \text{in} \ \Omega
\end{equation}
for each pair $u \in \USC(\overline{\Omega}), w \in \USC(\overline{\Omega})$ which are $\cF$-subharmonic, superharmonic respectively on $\Omega$. Namely, with $v:= -w \in \USC(\overline{\Omega}) \cap \wt{\cF}(\Omega)$, \eqref{CP1} is equivalent to 
\begin{equation}\label{CP2}
u + v \leq 0 \ \text{on} \ \partial \Omega \ \ \Rightarrow \ \ u + v \leq \ \text{in} \ \Omega, \quad \forall \, u \in \cF(\Omega), v \in \wt{\cF}(\Omega).
\end{equation}
However, if subharmonic addition \eqref{saf} holds, then the comparison principle \eqref{CP2} follows from the {\em Zero Maximum Principle on $\overline{\Omega}$ for dual cone subharmonics}: 
\begin{equation}\label{zmp}
z \leq 0 \ \text{on} \ \partial \Omega \ \Rightarrow \ z \leq 0 \ \text{in} \ \Omega, \quad \forall \, z \in \USC(\overline{\Omega}) \cap \wt{\cM}(\Omega).
\end{equation}
Morover, the validity of \eqref{zmp} for a dual monotonicity cone $\wt{\cM}$ on $\overline{\Omega}$ holds if
\begin{equation}\label{strict}
\mbox{$\exists \, \psi \in C^2(\Omega)$ strictly $\cM$-subharmonic in $\Omega$ \quad ($J^2_x \psi \in {\rm Int} \, \cM, \forall \, x \in \Omega$).}
\end{equation}
A classification of monotonicity cone subequations $\cM$ is known as is the determination for which domains $\Omega$ \eqref{strict} holds for a given $\cM$ (see Chapters 5 and 6 of \cite{CHLP23}). The reduction of the comparison principle \eqref{CP1} to the zero maximum principle for dual cone subequations \eqref{zmp} is called the {\em monotonicity-duality method} for proving the comparison principle in general potential theories. 

Notice that this method requires that jet addition \eqref{jaf} implies subharmonic addition \eqref{saf}. When $\cF$ has constant coefficients (constant fibers) this always happens as is shown in \cite{CHLP23}. In the variable coefficient case, an additional regularity condition on the subequation suffices.

\subsection{Fiberegularity.}\label{sec:fiberegularity} A subequation $\cF \subset \cJ^2(X)$ is {\em fiberegular} if the fiber map defined by
\begin{equation}\label{FR1}
\Theta: X \to \cK(\cJ^2) \ \ \Theta(x) := \cF_x \quad \forall \, x \in X
\end{equation}
is continuous where the closed subsets $\cK(\cJ^2)$ of $\cJ^2$ are equipped with the Hausdorff metric topology and $X$ inherits the euclidian topology of $\R^n$. When $\cF$ is $\cM$-monotone, fiberegularity has the more useful equivalent formulation: there exists $J_0 \in \Int \, \cM$ such that for each fixed $\Omega \subset \subset X$ and $\eta > 0$ there exists $\delta = \delta(\eta, \Omega)$ such that
\begin{equation}\label{FC}
x,y \in \Omega, |x-y| < \delta \ \ \Longrightarrow \ \  \Theta(x) + \eta J_0 \subset \Theta(y).
\end{equation}
Note that this property holds for each fixed $J_0 \in \Int \, \cM$ if it holds for one $J_0$ (see \cite{CPR23}) and that fiberegularity is uniform on bounded domains as the $\delta$ in \eqref{FC} is independent of $x,y \in \Omega$.

If $\cF$ is fiberegular and $\cM$-monotone, then \eqref{jaf} implies the subharmonic addition formula \eqref{saf} (see Theorem 4.4 of \cite{CPR23}) and hence the comparison principle on $\overline{\Omega}$ of \eqref{CP1} reduces to the zero maximum principle on $\overline{\Omega}$ of \eqref{zmp}. This is the {\em monotonicity-duality-fiberegularity method} for proving the comparison principle in general potential theories.

\section{Nonlinear elliptic differential operators}\label{sec:PEOs}

We now describe the class of degenerate elliptic operators that appears in the general correspondence principle. An operator $F \in C(\cG)$ where either
\begin{equation*}\label{case1}
\mbox{$\cG = \cJ^2(X)$ } \quad \text{({\em unconstrained case})}
\end{equation*}
or
\begin{equation*}\label{case2}
\mbox{$\cG \subsetneq  \cJ^2(X)$ is a subequation constraint set  \quad \text{({\em constrained case})}.}
\end{equation*}
is said to be {\em proper elliptic} if for each $x \in X$ and each $(r,p, A) \in \cG_x$ one has
\begin{equation}\label{peo}
F(x,r,p,A) \leq F(x,r + s, p, A + P) \ \quad \forall \, s \leq 0 \ \text{in} \ \R \ \text{and} \  \forall \, P \geq 0 \ \text{in} \ \cS(n).
\end{equation}
The pair $(F, \cG)$ will be called a {\em proper elliptic \footnote{In the unconstrained case, such operators are often refered to as {\em proper} operators (starting from \cite{CIL92}). We prefer to maintain the term  ``elliptic'' to emphasise the importance of the {\em degenerate ellipticity} ($\cP$-monotonicity in $A$) in the theory.}	(operator-subequation) pair}.

\begin{rem}(On proper ellipticity) A given operator $F$ must often be restricted to a suitable background subequation $\cG \subset \cJ^2(X)$ in order to satisfy the minimal monotonicty \eqref{peo} (the constrained case). The historical example is the Monge-Amp\`ere operator $F(D^2u) = {\rm det}(D^2u)$, where one restricts $F$ to the convexity subequation $\cG = \cP := \{ A \in \cS(n): A \geq 0 \}$. 
	\end{rem}

\noindent{\bf Admissible viscosity solutions:} Given a proper elliptic operator $F \in C(\cG)$, we say that 
\begin{itemize}
	\item[(sub)] $u \in \USC(X)$  is a {\em ($\cG$-admissible) subsolution of $F(J^2u)=0$ on $X$} if $\forall \, x \in X$ 
	\begin{equation}\label{AVSub}
	\mbox{$J \in J^{2, +}_{x}u$ \ \ $\Rightarrow$ \ \   $J \in \cG_x$ \ \ \text{and} \ \ $F(x, J) \geq 0$};
	\end{equation}
	\item[(super)] $u \in \LSC(\Omega)$ is a {\em ($\cG$-admissible) supersolution of $F(J^2u)=0$ on $X$} if $\forall \, x \in X$ 
	\begin{equation}\label{AVSuper}
	\mbox{$J \in J^{2, -}_{x}u$ \ \ $\Rightarrow$ \ \ either [ $J \in \cG_x$ and \ $F(x,J) \leq 0$\, ] \ or \ \ $J \not\in \cG_x$.}
	\end{equation}
\end{itemize}

In the unconstrained case where $\cG \equiv \cJ^2(X)$, the definitions are standard. In the constrained case where $\cG \subsetneq  \cJ^2(X)$, the definitions give a systematic way of doing of what is sometimes done in an ad-hoc way (see \cite{IL90} and \cite{Tr90}). In the constrained case (sub) says that subsolutions are also $\cG$-subharmonic and (super) says that $F(x,J) \leq 0$ for the lower test jets which lie in the constraint $\cG_x$.

\noindent {\bf Associated potential theory:} Given a proper elliptic operator $F \in C(\cG)$, a natural candidate for an associated subequation is
defined by the {\em correspondence relation}
\		\begin{equation}\label{CR0}
\cF = \{ (x,J) \in \cG: \ F(x,J) \geq 0 \}.
\end{equation}
The next section gives sufficient conditions under which $\cF$ defined by \eqref{CR0} for a proper elliptic pair $(F, \cG)$ is indeed a subequation and if the correspondence principle holds.

\section{The general correspondence principle}\label{sec:TCP}

The following general result is Theorem 7.4 of \cite{CPR23}. For the convenience of the reader, we reproduce the proof.

	\begin{thm}\label{thm:GCP}
	Given a proper elliptic operator $F \in C(\cG)$ and the associated constraint set $\cF$ defined by the correspondence relation
	\begin{equation}\label{CR}
	\cF = \{ (x,J) \in \cG: \ F(x,J) \geq 0 \}.
	\end{equation}
	Suppose that $\cF$ is a subequation and that it satisfies the compatibility condition 
	\begin{equation}\label{CC}
	\Int \, \cF = \{ (x,J) \in \cG: \ F(x,J) > 0 \},
	\end{equation} 
	which for a subequation is equivlent to $\partial \cF = \{ (x,J) \in\cG: \ F(x,J) = 0 \}$.
	Then the correspondence principle holds; that is: 
	\begin{equation*}\label{CP_sub}
u \ \text{is} \  \cF\text{-subharmonic} \ \text{in} \ X	\ \ \Leftrightarrow  \ \  	u \ \text{is a subsolution of} \ F(J^2u) = 0 \ \text{in} \ X 	.
	\end{equation*}
	\begin{equation*}\label{CP_super}
	u \ \text{is} \  \cF\text{-superharmonic} \ \text{in} \ X  \ \ \Leftrightarrow  \ \ u \ \text{is a supersolution  of} \ F(J^2u) = 0 \ \text{in} \ X.
	\end{equation*}
\end{thm}

\begin{proof}
By definition, $u \in \USC(X)$ is a ($\cG$-admissible) subsolution of $F(J^2u) = 0$ if for each $x \in X$:
	$$
	J^{2,+}_x u \subset \cG_x \ \ \text{and} \ \ F(x, J) \geq 0, \ \text{for each} \ J \in J^{2,+}_x u,
	$$
	which, for $\cF$ defined by the correspondence relation \eqref{CR}, is the $\cF$-subharmonicity of $u$.

By duality, $u \in \LSC(X)$ is $\cF$-superharmonic if and only if $-u \in \USC(X)$ is $\wt{\cF}$-subharmonic; i.e.\ for each $x \in X$:
	\begin{equation}\label{superharm}
J^{2,+}_x (-u)  \subset \wt{\cF}_x := - \left( {\rm Int} \, \cF_x \right)^c,			
	\end{equation}
	but $J^{2,+}_x (-u) = - J^{2,-}_x u$ and so \eqref{superharm} is equivalent to
	$$
	J^{2,-}_x u \subset \left( {\rm Int} \, \cF_x \right)^c,
	$$
	which, by the correspondence relation \eqref{CR} and the compatibility condition \eqref{CC},  means 
	$$
	\forall \, J \in J^{2,-}_x u: \quad \text{either} \ J \not \in \cG_x \quad \text{or} \quad [ J \in \cG_x \ \text{and} \ F(x,J) \leq 0],
	$$
	which is the definition of $u$ being a ($\cG$-admissible) supersolution of $F(J^2u) = 0$.
	\end{proof}

\subsection{Remarks and questions conerning the correspondence principle}

Two important remarks should be made concerning the general correspondence principle. The first remark highlights 
%\note{Slight reorganization of these two remarks for added clarity} 
an important consequence of the theorem and raises a question about the applicability of the theorem. The second remark begins an answer to the question.

\begin{rem}
	Theorem \ref{thm:GCP} says that the potential theory determined by any given subequation $\cF \subset \cJ^2(X)$ captures the subsolutions/supersolutions/solutions of the equation $F(J^2u) = 0$ for every proper elliptic operator $F \in C(\cG)$ which is compatible with $\cF$. This raises the question of whether given either a subequation $\cF$ or a proper elliptic operator $F$, can one determine the other so that the pair $(F, \cF)$ is compatible? 
\end{rem}

\begin{rem} Concerning the question raised in the remark above, we note that every subequation $\cF$ admits  associated compatible operators; one example is the {\em signed distance operator} defined by
	$$
	F(x,J) = \left\{ \begin{array}{ll} {\rm dist}(J, \partial \cF_x) & J \in \cF_x \\
	-{\rm dist}(J, \partial \cF_x) & J \in \cJ^2 \setminus \cF_x \end{array}	\right. .
	$$
On the other hand, given a proper elliptic operator $F \in C(\cG)$, the theorem requires that $\cF$ defined by the correspondence relation \eqref{CR} is indeed a subequation and that it also satisfies the compatibility condition \eqref{CC}.
\end{rem}

We formalize two questions concerning the applicability of Theorem \ref{thm:GCP}.

\begin{quest}\label{Q1}
Given a proper elliptic operator  $F \in C(\cG)$, what structural conditions on $(F, \cG)$ ensure that the constraint set $\cF$ defined by the correspondence relation
\begin{equation}\label{CR2}
\cF = \{ (x,J) \in \cG: \ F(x,J) \geq 0 \},
\end{equation}
is a subequation?
\end{quest}

\begin{quest}\label{Q2}
Given a proper elliptic operator  $F \in C(\cG)$, is it possibile to find a corresponding subequation that also satisfy the compatibility condition
\begin{equation}\label{CC2}
\Int \cF = \{ (x,J) \in \cG: \ F(x,J) > 0 \} \ \ \text{( equiv.} \ \partial \cF = \{ (x,J) \in\cG: \ F(x,J) = 0 \})?
\end{equation}
	\end{quest}

One answer to Question \ref{Q1} will be given in subsection \ref{sec:MMFR} and a discussion of Question \ref{Q2} will be given in subsection \ref{sec:exs}

\subsection{Fiberegular $\cM$-monotone subequations and operators}\label{sec:MMFR}

In this subsection, we give an answer to Question \ref{Q1}. Given a  given proper elliptic operator $F \in C(\cG)$, for each $x \in X, (r,p, A) \in \cG_x$ proper ellipticity means
	\begin{equation}\label{PEO}
F(x,r,p,A) \leq F(x,r + s, p, A + P) \ \quad \forall \, s \leq 0 \ \text{in} \ \R \ \text{and} \  \forall \, P \geq 0 \ \text{in} \ \cS(n).
	\end{equation}
The constraint set $\cF$ defined by \eqref{CR2} has fibers
	\begin{equation}\label{CR_x}
	\cF_x = \{ (r,p,A) \in \cG_x: \ F(x,r,p,A) \geq 0 \},
	\end{equation}
	which clearly satisfy the fiberwise subequation properties (P) and (N); that is,
	$$
(r,p,A) \in \cF_x \ \ \Rightarrow \ \ (r,p + s,A + P) \in \cF_x, \ \ \forall \, P \geq 0 \ \text{in} \ \cS(n), \forall \, s \leq 0 \ \text{in} \ \R.
	$$
The subequation topological stability properties (T)
	$$
\cF = \overline{\cF^{\circ}}, \ \  \left( \cF^{\circ} \right)_x = \left( \cF_x \right)^{\circ}, \ \ \cF_x = \overline{\left( \cF_x \right)^{\circ}}
	$$
can be proven under additional assumptions of $\cM$-monotonicity and fiberegularity of the pair $(F, \cG)$. More precisely, we assume that {\em $(F, \cG)$ is an $\cM$-monotone pair}, that is, there exists a monotonicity cone subequation $\cM \subset \cJ^2$ such that for each $x \in X$
	\begin{equation}\label{MM}
	\cG_x + \cM \subset \cG_x \quad \text{and} \quad F(x, J + J') \geq F(x,J) \quad \text{for each} \quad J \in \cG_x, J' \in \cM.
	\end{equation}
Next, we assume that {\em $(F, \cG)$ is fiberegular pair} in the sense that the following two conditions hold. First, $\cG$ is fiberegular; that is, the fiber map 
	\begin{equation}\label{FReg1}
	\Theta: X \to \cJ^2 \ \text{defined by} \ \Theta(x) = \cG_x \ \text{is continuous},
	\end{equation}
which has the equivalent formulation given in \eqref{FC} if $\cG$ is $\cM$-monotone. Second, assume that: $\exists \, J_0 \in {\rm Int} \, \cM$  such that $\forall \, \Omega \Subset X, \eta > 0 \ \exists \ \delta = \delta(\Omega, \eta)$ such that
	\begin{equation}\label{FReg2}
	F(y,J + \eta J_0) \geq F(x,J), \quad \forall \, J \in \cG_x, \forall \, x,y \in \Omega: |x-y| < \delta.
	\end{equation}
	
\begin{rem}[Fiberegular $\cM$-monotone pairs $(F,\cG)$] Every {\em gradient-free} proper elliptic pair $(F,\cG)$ is $\cM$-monotone with
$$
\cM := \cN \times \R^n \times \cP.
$$
Every $\cM$-monotone pair $(F,\cG)$	is fiberegular (the conditions \eqref{FReg1} and \eqref{FReg2} hold) in the unconstrained (convential) case where $\cG = \cJ^2$ or in the constant coefficient case.	
\end{rem}	
	
The following result (Theorem 7.11 of \cite{CPR23}) gives an answer to Question \ref{Q1}.

\begin{thm}\label{thm:FRMM}
	Let $(F, \cG)$ be a fiberegular $\cM$-monotone pair (i.e. \eqref{MM}, \eqref{FReg1} and \eqref{FReg2} hold). Then the constraint set defined by \eqref{CR2} 
	$$
	\cF = \{ (x,J) \in \cG: \ F(x,J) \geq 0 \}
	$$
	is a fiberegular $\cM$-monotone subequation. Moreover, the fibers are non empty if
	\begin{equation*}\label{NE}
	\Gamma(x) := \{ J \in \cG_x: \ F(x,J) = 0\} \neq \emptyset, \quad \forall \, x \in X.
	\end{equation*}
	Each fiber of $\cF$ is not all of $\cJ^2$ in the contrained case and also in the unconstrained case if one assumes
	\begin{equation*}\label{PF}
	 \{ J \in \cJ^2: \ F(x,J) < 0\} \neq \emptyset, \quad \forall \, x \in X.
	\end{equation*}
\end{thm}
We refer the reader to \cite{CPR23} for a complete proof, but a guide to the proof follows. As noted above, the subequation properties (P) and (N) for $\cF$ follow automatically since $\cM$-monotoncity implies the proper ellipticity of the pair $(F, \cG)$. The proof of the $\cM$-monotonicty of $\cF$ also follows easily by the definitions. The proof that the subequation property (T) holds for $\cF$ if $(F, \cG)$ is a fiberegular $\cM$-monotone pair makes use of a careful analysis of the relations between the triad of conditions contained in property (T) (see Proposition A.2 and Proposition A.5 of \cite{CPR23}). The proof of the fiberegularity of $\cF$ (in the non automatic conditions) is straightforward.

\begin{rem}[On $\cM$-monotonicity and fiberegularity]\label{rem:MM}
These two additional conditions are used here to show that the candidate constraint $\cF$ defined by the correspondence relation \eqref{CR2} is indeed a subequation, but it gives more. Namely, one also obtains the fiberegularity and $\cM$-monotonicity, which are used in the monotonicity-duality-fiberegularity method for proving potential theoretic comparison. Moreover, in light of the correspondence principle, the potential theoretic comparison principle (for $\cF$) passes to the operator theoretic comparison principle (for the PDE determined by the operator $F \in C(\cG)$).
	\end{rem}

\subsection{Finding compatible subequations}\label{sec:exs}
We close by adressing Question \ref{Q2} of whether given a proper elliptic operator $F \in C(\cG)$ can one find a subequation $\cF$ which is compatible with $F$? We will present three examples which illustrate that this can be easy, impossible or requires some work.

\begin{exe}[Perturbed Monge-Amp\'ere]\label{exe:PME} With $f \in C(X,\R)$ non-negative and $M \in C(X; \cS(n))$ consider the equation
\begin{equation}\label{PMAE}
{\rm det} \, \left( D^2u + M(x) \right) = f(x), \quad x \in X.
\end{equation}
In order to have degenerate ellipticity for $F$, it is necessary to restrict $F(x,A) := {\rm det} \, \left( A + M(x) \right) - f(x)$ to the subequation constraint set $\cG$ defined by
	$$
	\cG_x = \{ (r,p,A) \in \cJ^2:  \ A + M(x) \geq 0\}, \quad x \in X.
	$$
It is clear that the pair $(F, \cG)$ is $\cM$-monotone for the monotonicity cone subequation $\cM= \R \times \R^n \times \cP$ and that $(F,\cG)$ is fiberegular.

Hence $\cF$ defined by the correspondence relation \eqref{CR2}
	$$
	\cF_x = \{ (r,p,A) \in \cG_x: \ F(x,A) := {\rm det} \, \left( A + M(x) \right) - f(x) \geq 0\}, \quad x \in X
	$$
	is an $\cM$-monotone fiberegular subequation. Finally, it is easy to see that the compatibility condition \eqref{CC2} holds; so the correspondence principle holds and the monotonicity-duality-fiberegularity method applies.
\end{exe}

\noindent {\bf Note:} This ``easy'' example is also a good example of ``potential theory helping the operator theory'' in the following sense. In order to apply to prove the comparison principle or make use of Perron's method, conventional viscosity methods (see \cite{CIL92}) require $M$ to be the square of a Lipschitz matrix-valued map. We require only that $M$ is continuous. It follows that the comparison principle
holds $\forall \, \Omega \Subset X$ and Perron's method applies (if $\partial \Omega$ stricly convex and $C^2$). See Theorem 5.9 \cite{CP17} for details.

\begin{exe}[Optimal transport with directionality]\label{exe:OTE} With $f \in C(X), g \in C(\R^n)$ non-negative consider the equation
\begin{equation}\label{OTE}
g(Du) \det(D^2 u) = f(x), \quad x \in X.
\end{equation}
A natural compatible subequation is given fiberwise by 	
$$
	\cF_x = \{ (r,p,A) \in \cM(\cP):=\R \times \R^n \times \cP: \ F(x,p,A) \geq 0\}, \quad x \in X,
$$
with $F(x,p,A):= g(p) \, {\rm det} \, \left( A \right) - f(x)$.
The subequation axioms (P) and (N) and the compatibility condition \eqref{CC2} hold for $\cF$, but $\cM$-monotonicty \eqref{MM} and fiberegulairty \eqref{FReg1}-\eqref{FReg2} may fail for $(F, \cM(\cP))$. 

 For \eqref{MM}, assume that the gradient factor $g$ has the monotonicity property of {\em directionality}; that is, $\exists \, \cD \subset \R^n$ a closed convex cone with vertex at the origin  such that
	\begin{equation}\label{D}
	g(p + q) \ge g(p), \quad \forall \, p,q \in \cD.
	\end{equation}
	Then, restricting $F$ to the monotonicity cone with constant fibers $\cG_x= \cM(\cD,\cP):=\R \times \cD \times \cP$
	gives a $\cM(\cD,\cP)$-monotone pair $(F, \cG)$, where $\cG$ is trivially fiberegular (\eqref{FReg1} holds).

 The other part of fiberegularity \eqref{FReg2} holds with an additional {\em regularity property} on $g$: there exist $\bar q \in \Int \, \cD$ and $\omega : (0,\infty) \to (0,\infty)$ satisfying $\omega(0^+) = 0$ such that	
	$$
	g(p + \eta \bar q) \ge g(p) + \omega(\eta) \quad \text{for each} \ p,q \in \cD \ \text{and each} \ \eta > 0.
	$$
\end{exe}

\noindent{\bf Note:} The equation \eqref{OTE} describes the optimal transport plan from the source density $f$ to the target density $g$ (see \cite{DF14} or \cite{V} for a desciption of the problem).

\begin{exe}[Correspondence fails]\label{exe:CC_fail} Consider the constant coefficient gradient free equation
\begin{equation}\label{E}
	{\rm det}(D^2u) - u = 0.
\end{equation}
 $F(r,A) := {\rm det}(A) - r$ is proper elliptic when resricted to the monotonicity cone subequations
	$$
	\cG_1= \{ (r,A) \in \R \times \cS(n) : r \leq 0, A \geq 0\} \ \ \text{or} \ \ \cG_2= \{ (r,A) \in \R \times \cS(n) :  A \geq 0 \}.
	$$
	Clearly $(F, \cG_k)$ is an $\cM$-monotone pair with $\cM = \cG_k$ and is fiberegularity since it has constant coefficients. 

 However, the $\cM$-monotone fiberegular subequations $\cF_k$ defined by the correspondence relation \eqref{CR2}
	$$
	\cF_k := \{ (r,A) \in \cG_k: F(r,A) := {\rm det}(A) - r \geq 0\}
	$$
	have boundaries which contain the set 
	$$
	\{(r,A) \in \R \times \cS(n): r < 0 \ \text{and} \ A = 0\},
	$$
	whose harmonics are negative affine functions, which do not satisfy \eqref{E}; the correspondence principle fails. 
	\end{exe}
	
	\noindent{\bf Note:} The problem here is that the compatibility condition \eqref{CC2} fails
	$$
\partial \cF_k \supsetneq \{ (r,A) \in \cG_k: \ F(r,A) = 0 \}.
	$$

\end{document}